\documentclass[a4paper,12pt,reqno]{amsart}
\usepackage{amssymb}
\usepackage[margin=1in]{geometry}
\usepackage{marginnote}
\usepackage{mathtools}
\usepackage{datetime}
\usepackage{amscd}
\usepackage{amsfonts}
\usepackage{setspace}
\usepackage{version}
\usepackage[pdftex,bookmarks=true,bookmarksnumbered,colorlinks=true]{hyperref}
\usepackage[hyperpageref]{backref}
\hypersetup{
pdfauthor = {Stefanos Aivazidis},
pdftitle = {On the subgroup permutability degree of the simple Suzuki groups},
pdfkeywords = {subgroup permutability degree, Suzuki groups, finite simple groups},
			}
\mathtoolsset{showonlyrefs}
\numberwithin{equation}{section}
\theoremstyle{plain}
\newtheorem{theorem}[subsection]{Theorem}

\newtheorem{lemma}[subsection]{Lemma}
\newtheorem{corollary}[subsection]{Corollary}
\newtheorem{remark}[subsection]{Remark}
\newtheorem{definition}[subsection]{Definition}
\newtheorem{conjecture}[subsection]{Conjecture}

\newtheorem{claim}[subsection]{Claim}
\newtheorem{question}[subsection]{Question}
\newsavebox{\proofbox}
\savebox{\proofbox}{\begin{picture}(7,7)  \put(0,0){\framebox(7,7){}}\end{picture}}

\newcommand\E{\mathcal{E}}
\newcommand\U{\mathcal{U}}
\newcommand\K{\mathcal{K}}
\renewcommand\H{\mathcal{H}}
\newcommand\R{\mathcal{R}}
\newcommand\tr{\operatorname{Tr}}
\newcommand\F{\mathbb{F}}
\newcommand\C{\mathcal{C}}
\newcommand\GL{\operatorname{GL}}

\newcommand\N{\mathbb{N}}
\newcommand\G{\mathfrak{G}}
\newcommand\s{\mathfrak{s}}
\newcommand\f{\overline{f}}
\newcommand\p{\mathfrak{p}}

\newcommand\Aut{\operatorname{Aut}}
\newcommand\I{\operatorname{I}}
\newcommand\Sz{\operatorname{Sz}}
\newcommand\Gal{\operatorname{Gal}}
\newcommand\Syl{\operatorname{Syl}}
\newcommand\Per{\operatorname{Per}}
\newcommand\Hom{\operatorname{Hom}}
\renewcommand{\leq}{\leqslant}
\renewcommand{\geq}{\geqslant}

\newlength\sumd
\newtheoremstyle{named}{}{}{\itshape}{}{\bfseries}{.}{.5em}{\thmnote{#3's }#1}
\theoremstyle{named}

\begin{document}
\title[On the subgroup permutability degree of $\Sz(q)$]{On the subgroup permutability degree of the simple Suzuki groups}
\author{Stefanos Aivazidis}
\date{\today}
\thanks{The author acknowledges financial support from the N. D. Chrysovergis endowment under the auspices of the National Scholarships Foundation of Greece.}
\address{School of Mathematical Sciences, Queen Mary, University of London, London E1 4NS}
\email{s.aivazidis@qmul.ac.uk}
\maketitle
\begin{abstract}
We prove that the subgroup permutability degree of the 
simple Suzuki groups vanishes asymptotically. In the course 
of the proof we establish that the limit of the probability of a
subgroup of $\Sz(q)$ being a 2-group is equal to 1.
\end{abstract}
\section{Introduction}
Consider a finite group $G$ and subgroups $H$, $K$ of $G$. We say that $H$ and $K$ permute if $HK=KH$, and call $H$ a permutable (or quasi-normal) subgroup if $H$ permutes with every subgroup of $G$. A group $G$ is called quasi-Dedekind if all subgroups of $G$ are permutable. Recently T{\u{a}}rn{\u{a}}uceanu~\cite{tuarnuauceanu09} introduced the concept of subgroup permutability degree as the probability that two subgroups of 
$G$ permute
\begin{equation*} 
\p(G)	\coloneqq		\frac{\left| \left\{(H,K) \in \s(G) \times \s(G):HK=KH\right\}\right|}{\left|\s(G)\right|^2} 	= 		
\frac{1}{{\left|\s(G)\right|^2}}{\sum\limits_{H \leq G}{\left|\Per(H)\right|}}\label{initial},
\end{equation*}
where $ \Per ( H ) \coloneqq \left\{ K \leq G : HK = KH \right\} $, and $\s(G)$ is the set of subgroups of $G$. 
Thus $\p$ provides us with an arithmetic measure of how close $G$ is to being quasi-Dedekind.
This, we recall, is a property that lies strictly between the property of being abelian and that of being nilpotent, i.e.,
\begin{equation*} 
\text{abelian} \subsetneq \text{quasi-Dedekind} \subsetneq \text{nilpotent}.
\end{equation*}
Clearly an abelian group is quasi-Dedekind since all subgroups are normal,
thus permutable. The second containment follows from a celebrated 
result of Ore that permutable subgroups of finite groups are subnormal---in 
particular, a maximal subgroup of a finite quasi-Dedekind group 
is normal in the said group. In fact a finite group $G$ is quasi-Dedekind if and only 
if $G$ is a nilpotent modular group [Theorem 5.1.1,\cite{schmidt}]. 
We remind the reader that a group $G$ is called modular if its subgroup lattice 
is modular, that is, if $\langle H, K \cap L \rangle = \langle H, K \rangle \cap L$ 
for all subgroups $H,K,L$ of $G$ such that $H \leq L$. Thus one has the containments
\begin{equation*} 
\text{abelian} \subsetneq \text{quasi-Dedekind $\leftrightarrow$ nilpotent modular} \subsetneq \text{nilpotent}.
\end{equation*}
It therefore seems natural to speculate that simple groups are quite far from being quasi-Dedekind.
The main result of the present paper serves as a testament to this intuition by focussing on
the family of simple Suzuki groups. Indeed, we shall prove the following theorem.
\begin{theorem}\label{main}
The subgroup permutability degree of \hspace{0,01 mm} $\Sz\left(2^{2n+1}\right)$ vanishes asymptotically, i.e., 
\begin{equation*}
\lim\limits_{n\to +\infty}\mathfrak{p}\left(\Sz \left(2^{2n+1}\right)\right)=0.
\end{equation*}
\end{theorem}
\noindent The proof of Theorem \ref{main} is carried out in two steps. First we offer a criterion 
for the vanishing of the subgroup permutability degree of an infinite family of groups 
under a set of suitable hypotheses. The second step consists in establishing that 
$\Sz(q)$ satisfies each of these hypotheses. The former is straightforward and it is 
precisely the content of section~\ref{3}. The latter is more involved and it will occupy 
the remainder of the paper, which is organised as follows.
 
In section~\ref{4} we outline the subgroup structure of $\Sz(q)$ 
with particular emphasis on the structure of a Sylow 2-subgroup 
$P$ and that of its normaliser. 
In section~\ref{5} we discuss a method of Hulpke for determining 
the conjugacy classes of subgroups of a soluble group, and apply 
this method to $P$ in order to obtain bounds for $\left| \s( P ) \right|$. 
In subsection \ref{normaliser} we do the same for the normaliser. 
In section~\ref{6} we use standard techniques from calculus to 
compare the number of subgroups of the normaliser with that of $P$, 
and find that these are asymptotically equal. 
Finally, we prove in section~\ref{7} that 2-subgroups dominate the 
subgroup lattice of $\Sz(q)$; this is the only nontrivial condition of our 
criterion in section~\ref{3} that actually requires proof, as will soon 
become apparent to the reader. 
In section~\ref{8} we conclude our exposition with a list of questions 
and problems that offer potential for future research.
\subsection{Notation}
For the convenience of the reader we recall standard 
notation outside the realm of algebra, and explain 
notational conventions on the part of the author
that will be used throughout the paper. 
\begin{enumerate}
\item Let $n \in \mathbb{N}$. Then $d(n)$ is the number of 
divisors of $n$, and $\omega(n)$ is the number of distinct 
prime divisors of $n$. 
\item For the sequences $\left\{f_n\right\}$, $\left\{g_n\right\}$, 
$g_n \neq 0$, we will write  $f_n \sim g_n$ if 
$\lim\limits_{n \rightarrow \infty}f_n/g_n=1$.
\item Suppose that $G$ is a group, and let $x,y,g \in G$. 
We shall write the conjugate of $x$ with respect to $g$ as $gxg^{-1}$, 
and the commutator of $x$, $y$ (in that order) as $xyx^{-1}y^{-1}$. 
\item We say that the $p$-group $P$ is a special $p$-group if 
either it is elementary abelian, or if $P' = Z( P ) = \Phi ( P )$ is 
elementary abelian. For the (not necessarily special) $p$-group 
$P$ we shall write $\mho ( P )$ for the subgroup generated by 
the $p$-powers of elements of $P$.
\item If $G_1, G_2$ are groups, then $\Hom \left( G_1, G_2 \right)$ 
is the set of all homomorphisms $G_1 \to G_2$. 
\item If $V,W$ are vector spaces over the field $\F$, then 
$\mathcal{L} \left(V,W \right) \left( \F \right) = \mathcal{L} \left(V,W \right)$ 
stands for the vector space of all linear transformations $V \to W$. 
\item Let $\F_q$ be the finite field with $q=p^n$ elements, for some 
prime $p$ and some $n \in \N$. We shall write $V(k,q)$ for the vector 
space $\F_q^k$. 
\end{enumerate}
\section{Main Lemma}\label{3}
Let us now focus on the criterion for the vanishing of the subgroup permutability 
degree that we mentioned earlier. In general, working 
with the definition of $\p$ seems difficult---there is usually little or no insight when 
two randomly chosen subgroups of a group permute, perhaps because they may permute 
for a variety of reasons. Even if one were only to consider groups for which subgroup 
permutability is reduced to a more manageable property\footnote{This is for example the case 
with the so-called equilibrated groups of Blackburn et al. \cite{equilibrated}.},
one should still be able to say something useful about the behaviour of
the various sums that would ultimately appear in the resulting
expression for $\p$. 

One should therefore ask if perhaps ``most" subgroups
of the group in question are of a particular type, and if so, whether subgroup 
permutability between those subgroups can be decided effectively. The simplest
case arises when $p$-subgroups dominate the subgroup lattice for some prime 
$p$ dividing the order of the group, and when in addition the Sylow $p$-subgroups 
intersect trivially. In this case it suffices to only check permutability between subgroups 
of the same Sylow $p$-subgroup. The following lemma makes this precise.
\begin{lemma} \label{TI}
Let $\left\{\G_n\right\}_{n=1}^{+\infty}$ be a family of finite groups such that 
$p \mid \left| \G_n\right|$ for some fixed prime $p$ and for all $n \in \N$, 
satisfying the conditions
\begin{enumerate}
\item \label{1}the Sylow $p$-subgroups of $\G_n$ intersect trivially for all $n \in \N$, 
\item \label{2}$\lim\limits_{n \to +\infty}\left| \Syl_p(\G_n) \right|=+\infty$, and
\item \label{3}$\lim\limits_{n \to +\infty}\dfrac{\left| \E_n \right|}{\left|\s(\G_n)\right|}=1$,
\end{enumerate}
where 
\begin{equation*}
\E_n \coloneqq \left\{ H \leq \G_n : \left| H \right|=p^k\,\ \text{for some} \,\ k \in \N \right\}=\bigcup\limits_{P \in \Syl_p(\G_n)}\s( P).
\end{equation*}
Then $\lim\limits_{n \to +\infty}\p(\G_n)=0.$
\end{lemma}
\begin{proof}
Define the map $f:\s(\G_n) \times \s(\G_n) \to \left\{0,1\right\}$ via the rule
\begin{equation*}
\left(H_i,H_j\right) \mapsto
\begin{cases} 	1,& \text{if $H_iH_j=H_jH_i$,}\\ 
			0,& \text{otherwise,}
\end{cases}
\end{equation*}
and observe that $f$ is symmetric in its arguments. Thus
\begin{eqnarray*}
\sum\limits_{H \leq\G_n}{\left|\Per(H)\right|}	&=&		\sum\limits_{ X_i,X_j \in \E_n}f\left(X_i,X_j\right)+
											2\sum\limits_{\substack{X_i \in \E_n\\  Y_j \in \E_n^c}}f\left(X_i,Y_j\right)\\                                                                                                                            
									&+&		\sum\limits_{ Y_i,Y_j \in \E_n^c}f\left(Y_i,Y_j\right) \\
									&\le&        \sum\limits_{ X_i,X_j \in \E_n}f\left(X_i,X_j\right)+
											2\sum\limits_{\substack{X_i \in \E_n\\  Y_j \in \E_n^c}}1+
											\sum\limits_{ Y_i,Y_j \in \E_n^c}1 \\
									&=&		\sum\limits_{ X_i,X_j \in \E_n}f\left(X_i,X_j\right)+
											2\left|\E_n\right|\left|\E_n^c\right|+\left|\E_n^c\right|^2 \\
									&=&		\sum\limits_{ X_i,X_j \in \E_n}f\left(X_i,X_j\right)+
											\left|\s(\G_n)\right|^2-\left|\E_n\right|^2.
\end{eqnarray*}
Divide by $\left|\s(\G_n)\right|^2$ both sides to deduce that
\begin{equation}\label{comeback}
\p(\G_n) \leq 1-\frac{\left|\E_n\right|^2}{\left|\s(\G_n)\right|^2}+\frac{\sum\limits_{ X_i,X_j \in \E_n}f\left(X_i,X_j\right)}{\left|\s(\G_n)\right|^2}.
\end{equation}
Now let $X_i,X_j \in \E_n$. We claim that if $X_iX_j$ is a subgroup of 
$\G_n$, then both $X_i,X_j$ belong to the same Sylow $p$-subgroup. 
To see this, let $P \in \Syl_p(\G_n)$. Then there exist elements 
$g_i,g_j$ of $\G_n$ such that $X_i \leq P^{g_i}$, and $X_j \leq P^{g_j}$. Since
\begin{equation*}
\left|X_i X_j \right| = \frac{\left| X_i \right| \left| X_j \right|}{\left|X_i \cap X_j \right|},
\end{equation*}
and because $X_i$, $X_j$ are $p$-groups, so is $X_iX_j$. Hence there exists an 
element $g_k \in \G_n$ such that $X_iX_j \leq P^{g_k}$. Notice that 
$X_i \leq P^{g_i}$ and $X_i \leq X_iX_j \leq P^{g_k}$. Thus 
$P^{g_i} \cap P^{g_k}\geq X_i>1$. Since distinct Sylow $p$-subgroups of $\G_n$ 
intersect trivially, we deduce that $P^{g_i} = P^{g_k}$. Similarly 
$P^{g_j} \cap P^{g_k}\geq X_j>1$, and this forces $P^{g_j} = P^{g_k}$ 
for the same reason. We conclude that $P^{g_i} = P^{g_j}$, thus both 
$X_i$ and $X_j$ are subgroups of the same Sylow $p$-subgroup, as required. 

\noindent Now let $\Syl_p(\G_n) = \left\{P^{g_i} \mid 0 \leq i \leq \left| \Syl_p(\G_n) \right|\right\}$. By dint of 
the above observation we may thus write
\begin{eqnarray*}
\sum\limits_{ X_i,X_j \in \E_n}f\left(X_i,X_j\right)		&=&                \sum\limits_{k=1}^{\left| \Syl_p(\G_n) \right|}\sum\limits_{ X_i,X_j \in P^{g_k}}f\left(X_i,X_j\right)\\
											&\leq&            \sum\limits_{k=1}^{\left| \Syl_p(\G_n) \right|}\sum\limits_{ X_i,X_j \in P^{g_k}}1\\
											&=&                \sum\limits_{k=1}^{\left| \Syl_p(\G_n) \right|}\left(\left|\s\left(P^{g_k}\right)\right|-1\right)^2\\
											&=&                \left| \Syl_p(\G_n) \right| \left( \left|\s ( P ) \right|-1\right)^2.                                                                       
\end{eqnarray*}
On the other hand we have
\begin{equation*}
\left|\E_n\right|^2                =                \left| \Syl_p(\G_n) \right|^2\left( \left| \s( P ) \right|-1 \right)^2.
\end{equation*}                                                        
Hence
\begin{eqnarray*}
0 \leq        \frac{\sum\limits_{ X_i,X_j \in \E_n}f\left(X_i,X_j\right)}{\left|\s(\G_n)\right|^2}        &\leq&	\frac{\sum\limits_{ X_i,X_j \in \E_n}f\left(X_i,X_j\right)}{\left|\E_n\right|^2}\\
                                                                                                                                                &\leq&	\frac{\left| \Syl_p(\G_n) \right|\left( \left| \s( P ) \right|-1 \right)^2}{\left| \Syl_p(\G_n) \right|^2\left( \left| \s( P ) \right|-1 \right)^2}\\
                                                                                                                                                &=&		\frac{1}{\left| \Syl_p(\G_n) \right|},
\end{eqnarray*}
from which we see that
\begin{equation*}
\lim\limits_{n\to +\infty}\frac{\sum\limits_{ X_i,X_j \in \E_n}f\left(X_i,X_j\right)}{\left|\s(\G_n)\right|^2}=0,
\end{equation*}
since $\lim\limits_{n \to +\infty}\left| \Syl_p(\G_n) \right|=+\infty$, thus $\lim\limits_{n \to +\infty}{\left|\Syl_p(\G_n)\right|}^{-1}=0$. Also
\begin{equation*}
\lim\limits_{n\to +\infty}\frac{\left|\E_n\right|^2}{\left|\s(\G_n)\right|^2}=1,
\end{equation*}                                                                
since $\lim\limits_{n\to +\infty}\frac{\left|\E_n\right|}{\left|\s(\G_n)\right|}=1$, by hypothesis.
Taking limits in \eqref{comeback} yields
\begin{equation*}
0 \leq \lim\limits_{n\to +\infty}\p(\G_n) \leq \lim\limits_{n\to +\infty}\left(1-\frac{\left|\E_n\right|^2}{\left|\s(\G_n)\right|^2}+\frac{\sum\limits_{ X_i,X_j \in \E_n}f\left(X_i,X_j\right)}{\left|\s(\G_n)\right|^2}\right)=0,
\end{equation*}                                                                
thus concluding the proof.
\end{proof}

\section{The subgroup structure of $\Sz(q)$}\label{4}
The discussion in this section follows closely that of 
Nouacer~\cite{nouacer}, and Berkovich and Janko~\cite{bj3}, \S 105. Let $\F_q$ be the 
finite field with $q \coloneqq 2^{2n+1}$ elements and 
set $\theta \coloneqq 2^{n+1}$. The map
$\overline{\theta} : x \mapsto x^{\theta}$ is an 
automorphism of the field and, in fact, generates the 
cyclic group $\Gal \left( \F_q/ \F_2 \right)$. This is 
because $\left| \Gal \left( \F_q/ \F_2 \right)\right|=2n+1$ 
and $\overline{\theta}$ acts as a ``square root" of the 
Frobenius automorphism $\phi$, that is, 
$x^{\theta^2} = x^2$ for all $x \in \F_q$, hence both 
$\overline{\theta}$ and $\phi$ have the same order 
in $\Gal \left( \F_q/ \F_2 \right)$.

\begin{definition}[Suzuki group]\label{suz-def} 
Suppose that $a, b, \in \F_q$ and $\lambda \in \F_q^{\times}$.
Define $4 \times 4$ matrices over $\F_q$ by
\begin{align*} 
S(a,b) &\coloneqq 
\begin{pmatrix} 
1 							& 0 							& 0 									& 0 								\\ 
a 							& 1 							& 0 									& 0 								\\ 
b 							& a^{\theta} 					& 1 									& 0 								\\ 
a^{2+\theta} + ab + b^{\theta} 		& a^{1+\theta}+b 				& a 									& 1
\end{pmatrix},\\
C(\lambda) &\coloneqq 
\begin{pmatrix} 
{\lambda}^{1+\frac{\theta}{2}} 		& 0 							& 0 									& 0	 							\\ 
0 							& {\lambda}^{\frac{\theta}{2}} 		& 0 									& 0 								\\ 
0 							& 0 							& {\lambda}^{-\frac{\theta}{2}} 			& 0 								\\ 
0 							& 0 							& 0 									& {\lambda}^{-1-\frac{\theta}{2}}
\end{pmatrix}, 
T \coloneqq 
\begin{pmatrix} 
0 	& 0 	& 0 	& 1 \\ 
0 	& 0 	& 1 	& 0 \\ 
0 	& 1 	& 0 	& 0 \\ 
1 	& 0 	& 0 	& 0 
\end{pmatrix}.
\end{align*}
The Suzuki group $\Sz(q)$ is defined to be the following subgroup of $\GL_4(q)$
\begin{equation*} 
\Sz(q) \coloneqq \left\langle S(a,b), C(\lambda), T \mid a,b \in \F_q, \lambda \in \F_q^{\times} \right\rangle.
\end{equation*}
\end{definition}
In this notation, the set $P \coloneqq \left\{ S(a,b) \mid (a,b) \in \F_q^2 \right\}$
is a Sylow 2-subgroup of $\Sz(q)$. In fact, $P \cong \left(\F_q^2,\ast\right)$, where $\ast$
is defined via the rule \[(a_1,b_1)\ast(a_2,b_2) = (a_1+a_2,b_1+b_2+a_1a_2^{\theta}),\]
the implicit isomorphism being $S(a,b) \mapsto (a,b)$. This writing of $P$ as a  
direct product endowed with a ``twisted" multiplication is particularly convenient, 
as it captures the essential information contained within each matrix while avoiding 
the cumbersome matrix notation. 

\noindent Now notice that $(0,0)$ is the identity element, and $(a,b)^{-1} = (a,b+a^{1+\theta})$, hence
\begin{equation} \label{commutators}
\left[(a_1,b_1) , (a_2,b_2) \right] = \left(0,a_1a_2^{\theta}+a_2a_1^{\theta} \right). 
\end{equation}
If either $a_1=0$ or $a_2=0$, then $\left[(a_1,b_1) , (a_2,b_2) \right]=(0,0)$. Moreover
$(0,b_1) \ast (0,b_2) = (0,b_1+b_2)$ and $(0,b)^2=(0,0)$, thus $\left\{ (0,b) : b \in \F_q \right\} \leq Z$. In fact,
equality occurs here. For suppose that $(a_1,b_1) \in Z$. Then $\left(0,a_1a_2^{\theta}+a_2a_1^{\theta} \right) = (0,0)$
for all $a_2 \in \F_q$, thus $a_1a_2^{\theta}=a_2a_1^{\theta}$, since $\mathrm{char} \F_q = 2$.
Because $n \geq 1$, we may choose $a_2 \in \F_q \setminus \left\{0,a_1\right\}$.
Therefore $\left( a_1a_2^{-1} \right)^{\theta} = a_1a_2^{-1}$, i.e., the element $a_1a_2^{-1}$
is a fixed point of the automorphism $\overline{\theta}$. Since
$\left\langle \overline{\theta} \right\rangle = \Gal \left( \F_q/ \F_2 \right)$, the fixed
points of $\overline{\theta}$ are precisely the elements of the prime subfield $\F_2=\left\{0,1\right\}$.
Hence $a_1a_2^{-1} = 0$, that is $a_1=0$. Thus $Z \leq \left\{ (0,b) : b \in \F_q \right\}$, which
establishes the claim. We deduce that the centre of $P$ is an elementary abelian group, isomorphic
to the additive group of the field. 

From \eqref{commutators} it is clear that $P' \leq Z$, since all commutators are central, hence $P/Z$
is abelian. Moreover $(a,b)^2=(0,a^{1+ \theta}) \in Z$, thus all squares are central as well.
In view of $\left| P/Z \right| = \left| Z \right|$, we infer that $P/Z \cong Z$.

As all squares lie in the centre, clearly $\mho( P) \leq Z$ holds. Consider an arbitrary element $(0,b) \in Z$,
and notice that the map $x \mapsto x^{1+\theta}$ is a bijection of the field $\F_q$, since 
\begin{equation}\label{gcd} 
\gcd \left( q-1,1+\theta \right) = \gcd \left( 2^{2n+1}-1,1+2^{n+1} \right) =1.
\end{equation}
Thus there exists a unique element $a_b \in \F_q$ such that $a_b^{1+\theta} = b$. 
Therefore $(0,b) = (a_b,b)^2 \in \mho( P)$, which proves that $\mho( P) = Z$. Also 
$\Phi( P) = \mho( P)P'$ when $P$ is a $p$-group\footnote{See Rotman~\cite{rotman}, Theorem 5.48.};
since $P' \leq Z$ so $\Phi( P) = Z$.

Proving that $P'$ and $Z$ actually coincide is not difficult. The 
multiplicative group of the field is a subgroup of $\Aut (P )$, and 
acts transitively on the non-identity elements of $Z$, as we shall
shortly see. Since $P'$ is a characteristic subgroup of $P$, it is
invariant under the action via automorphisms of $\F_q^{\times}$.
The claim now follows from $P' \leq Z$, which we already know.
In spite of the simple argument above, we offer an alternative 
proof that is essentially due to Isaacs. It is more direct and, if 
modified appropriately, works equally well in a more general
setting.
\begin{claim}\label{isaacs}
Let $P \in \Syl_2 \left( \Sz \left(2^{2n+1}\right) \right)$, $n \geq 1$. Then $P'=Z$.
\end{claim}
\begin{proof}\textit{(Isaacs)}
It is sufficient to show that the subgroup of the additive
group of $\F_q$ generated by the elements of the form
$x y^{\theta} + x^{\theta} y$ is the whole group.
Taking $x = 1$ and letting $y$ vary over $\F_q$ gives all elements of the form
$y^{\theta} + y$. This set is actually a subgroup since the map
$y \mapsto y^{\theta} + y$ is an additive homomorphism. Furthermore,
the kernel of this homomorphism is the prime subfield $\F_2$, 
and thus by taking $x = 1$, we get a subgroup of $\F_q$ of index $2$. 
In fact, every member of this subgroup has trace
zero, where the trace of an element $t \in \F_q$ is understood to be
$\tr\left(t\right) = \sum_{\sigma \in \langle \overline{\theta} \rangle} t^\sigma$.
It is known that the trace map maps $\F_q$ onto the prime subfield, so the 
kernel of the trace is a subgroup of index $2$. Thus taking $x = 1$ yields 
exactly the elements with trace zero.

It suffices now to find $x$ and $y$ such that $x y^{\theta} + x^{\theta} y$
does not have trace zero. It will follow that the group generated
by the elements of the form $x y^{\theta} + x^{\theta} y$ is the whole
of $\F_q$.
Now in general, $\tr(t) = \tr ( t^{\theta} )$, so
$ \tr\big(x y^\theta + x^\theta y\big) =
\tr\big(x y^\theta\big) + \tr\big(x^\theta y\big) =
\tr\big(x^\theta y^{{\theta}^2}\big) + \tr\big(x^\theta y\big) =
\tr\big(x^\theta\big(y^{{\theta}^2} + y\big)\big)$.
Since $q \geq 8$, ${\theta}^2$ is not the identity automorphism, so
choose $y$ so that $y^{{\theta}^2} + y \neq 0$, and write $c$ to
denote this nonzero element. It suffices now to find $x$ such that
$\tr\left(c x^\theta\right) \neq 0$. 
As $x$ varies over $\F_q$, the element  $c x^{\theta}$ runs over all of $\F_q$,
so for some value of $x$ we get an element with nonzero trace.  
This completes the proof.
\end{proof}

Notice that a subgroup $H \leq P$ 
either contained in $Z$, or containing $Z$ is normal in~$P$. The first
assertion is clear, while the second assertion follows from 
$h^g = [g,h]h$ being an element of $H$ for all $g \in P$, as
$[g,h] \in P' = Z$. We collect what we have established so far.

\textit{The group $P$ is a special 2-group of exponent 4 and class 2, 
with the property that $P/Z \cong Z$.}

\begin{remark}\label{3.2}
The Sylow 2-subgroups of $\Sz(q)$ arise as special cases in Higman's more general 
theory of so-called Suzuki 2-groups\footnote{See Higman~\cite{higman} for 
the original paper that introduces them (the groups $P$ appear 
as $A_2(n, \theta)$ therein), or Huppert and Blackburn~\cite{hb2}, Chapter VIII, \S 7 for 
a definitive account.}, i.e., nonabelian 2-groups with more than one 
involution, admitting a cyclic group of automorphisms which permutes 
their involutions transitively. The purpose of the first joint condition is 
to avoid considering known (and well understood) families of groups, 
such as elementary abelian, cyclic or generalised quaternion, which also 
have cyclic groups of automorphisms acting transitively on their 
involutions (in the elementary abelian case these are known as Singer cycles).
\end{remark}
Let us now consider the group $C \coloneqq \left\{ C(\lambda) : \lambda \in \F_q \right\}$.
This is a cyclic group, generated by $C(\lambda^{\ast})$, where $\lambda^{\ast}$
is any primitive element of $\F_q$. It is clearly isomorphic to the multiplicative group of the field,
where $\lambda \mapsto C(\lambda)$ establishes the said isomorphism, and acts via conjugation
on the Sylow 2-subgroup $P$. Since
\[
\lambda \cdot \left(a,b\right) = \left(a,b\right)^{\lambda} = \left( \lambda a , \lambda^{1+ \theta} b \right),
\]
and in view of \eqref{gcd}, the action of $C$ on the nonidentity elements 
of both $Z$ and $P/Z$ is regular. In fact, the action on $P$ is via automorphisms since
\begin{align*}
\lambda \cdot \left(a_1,b_1\right)\left(a_2,b_2\right) 	&= 	\lambda \cdot \left(a_1+a_2 , b_1+b_2+a_1a_2^{\theta}\right) 								\\
											&= 	\left(\lambda a_1 + \lambda a_2 ,\lambda^{1+ \theta} b_1 + \lambda^{1+ \theta} b_2 + \lambda a_1\left(\lambda a_2\right)^	{\theta}\right)																		\\
											&= 	\left( \lambda a_1 , \lambda^{1+ \theta} b_1 \right)\left( \lambda a_2 , \lambda^{1+ \theta} b_2 \right)	\\
											&=	\left( \lambda \cdot \left(a_1,b_1\right) \right)\left( \lambda \cdot \left(a_2,b_2\right) \right).
\end{align*}
The group $P \rtimes C$ is a Frobenius group with Frobenius kernel $P$ and Frobenius 
complement $C$. It is the normaliser of $P$ and is maximal in $\Sz(q)$. The maximal 
subgroups of $\Sz(q)$ are (up to conjugacy)\footnote{See Wilson~\cite{wilson}, \S 4.2.3., 
or the original source~\cite{suzuki62}, \S 15.}
\begin{enumerate}
\item the normaliser $P \rtimes C$ of a Sylow 2-subgroup $P$,
\item $Sz(q_0)$, where $q = q_0^r$, $r$ is prime, and $q_0 > 2$,
\item $D_{2(q-1)}$,
\item $C_{q+\theta +1} \rtimes C_4$,
\item $C_{q-\theta +1} \rtimes C_4$.
\end{enumerate}

\section{Conjugacy classes of complements and 1-cohomology}\label{5}
In this section we shall discuss an application of Hulpke's method for finding the 
conjugacy classes of subgroups of a soluble group to a Sylow 2-subgroup $P$ 
of $\Sz(q)$. The reader is referred to Hulpke~\cite{hulpke} for a detailed exposition of said 
method, and in particular section 3, Lemma~3.1. 

Consider a subgroup $H$ of $P$ and observe that $H \cap Z$ is central in $P$, 
thus normal in all subgroups of $P$ that contain it. Since $Z \lhd P$, the group 
$HZ$ is defined and is normal in $P$ from the discussion preceding Remark~\ref{3.2}, 
thus both quotient groups $Z \big/ H \cap Z$, $HZ \big/ H \cap Z$ are defined as well. 
In fact, $Z \big/ H \cap Z$ is a subgroup of $HZ \big/ H \cap Z$, and 
\begin{equation*}
HZ \big/H \cap Z \Big/ Z \big/H \cap Z \cong HZ\big/Z \cong H\big/H \cap Z.
\end{equation*}
Since $Z\big/ H \cap Z$ and $H \big/ H \cap Z$ intersect trivially, we see that $H \big/ H \cap Z$ is a
complement to $Z \big/ H \cap Z$ in $HZ \big/ H \cap Z$. Now let $H_1$, $H_2$ be a pair of subgroups 
of $P$. We observe the following.
\begin{lemma}\label{remark}
The subgroup $H_1$ is conjugate to $H_2$ if and only if $H_1 \big/ H_1 \cap Z$ is conjugate to $H_2 \big/ H_2 \cap Z$.
\end{lemma}
\begin{proof}
Suppose first that $H_2=H_1^g$ for some $g \in P$. Then 
$H_2 \cap Z = H_1^g \cap Z = H_1^g \cap Z^g = \left( H_1 \cap Z\right)^g$, 
thus 
$H_2 \big/ H_2 \cap Z = H_1^g  \big/ \left( H_1 \cap Z \right)^g = \left( H_1/H_1 \cap Z \right)^{\overline{g}}$.
Conversely, assume that $H_2 \big/ H_2 \cap Z=\left( H_1 \big/ H_1 \cap Z \right)^{\overline{g}}$ 
for some $\overline{g} \in \overline{P}$. Since 
$\left( H_1 \big/ H_1 \cap Z \right)^{\overline{g}} = H_1^g \big/ \left( H_1 \cap Z \right)^g$, 
we deduce that $H_2=H_1^g$.
\end{proof}

Let us now consider a set of representatives for the conjugacy classes of subgroups of $P$ that contain $Z$, say $\K$,
and a set of representatives for the conjugacy classes of subgroups of $Z$, say $\H$.
Evidently $\H$ is just the set of subgroups of $Z$, while the members of $\K$ are the full preimages of $\s\left( P/Z \right)$.
\begin{lemma}
Let $\K,\H$ be as above. For each $K \in \K$, $H \in \H$ denote by $\U_{K,H}$ the full preimages of a set of representatives for the 
$P$-classes of complements to $Z/H$ in $K/H$. Then 
\begin{equation}
\C = \bigcup\limits_{K \in \K} \bigcup\limits_{H \in \H}\U_{K,H} \label{doubleunion}
\end{equation}
is a set of representatives for the $P$-classes of subgroups of $P$.
\end{lemma}
\begin{proof}
Consider a subgroup $L$ of $P$, and let $K=\langle L,Z \rangle = LZ$, $H=L \cap Z$.
Then $L/H$ is a complement to $Z/H$ in $K/H$, thus $L$ is conjugate to a 
member of $\U_{K,H}$. Conversely, the proof of Lemma \ref{remark} shows that $L$ can be 
conjugate to at most one group from $\C$.
\end{proof}

We note that the above lemma does not tell us for which pairs of subgroups $\left( K,H \right)$ 
the set $\U_{K,H}$ is nonempty; only that, by considering all such pairs, we will end up 
with a complete list for the conjugacy classes of subgroups of $P$. We address this issue in the
following lemma, but we hasten to inform 
the reader that a method which treats the general case is available in Celler et. al.~\cite{celleretal}. 
\begin{lemma}\label{frattini}
Suppose that $Z \leq K \leq P$, and let $H$ be a central subgroup of $P$.
Then $Z/H$ has a complement in $K/H$ if and only if $K/H$ is elementary 
abelian. If such a complement does exist, then $ \left| \Phi(K) \right| \geq \left| K : Z \right|.$
\end{lemma}
\begin{proof}
Recall that $K/H$ is elementary abelian if and only if $\Phi(K) \leq H$,
since the Frattini subgroup of a finite $p$-group is the unique normal subgroup 
of said group minimal with the property that the quotient is elementary abelian.

Now notice that one direction of the first claim follows 
immediately. In an elementary abelian group all subgroups are direct summands, so if $K / H$ 
is elementary abelian, then $Z / H$ is complemented.

Conversely, suppose that $C / H$ is a complement to $Z / H$ in $K / H$. Let us first note that since 
$C / H$ is a complement, 
\begin{equation*}
C / H \cong K / H \Big/ Z / H \cong K/Z.
\end{equation*} 
However, since $K/Z$ is elementary abelian, $C/H$ is elementary abelian as well, 
thus $\Phi( C) \leq H$. Moreover, since 
\begin{equation*}
\left(Z / H\right)\left(C / H\right)=K / H,
\end{equation*} 
we see that $ZC = K$. Therefore $K' = \left(ZC\right)' = Z'C' = C',$ and 
$\mho(K) = \mho \left(ZC\right) = \mho( C)$, since $Z$ is central and elementary abelian.
Hence
\begin{equation*}
\Phi(K) = K' \mho(K) = C' \mho( C) = \Phi( C) \leq H.
\end{equation*}
We deduce that $K / H$ is elementary abelian and this settles the first claim.

In proof of the second claim, we observe that the inequality $ \left| \Phi(K) \right| \geq \left| K : Z \right|$ is equivalent to $\left| Z \right| \geq  \left| K : \Phi(K) \right|$. 
Recall that $Z = \Phi( P)$ and that $P/Z \cong Z$. It is therefore sufficient to establish that 
$ \left|P : \Phi( P) \right| \geq \left| K : \Phi(K) \right|$. However, by Burnside's Basis Theorem, the rank of 
$P / \Phi( P)$ is the size of a minimal generating set for $P$. Evidently the subgroup $K$ 
requires at most as many generators as $P$ does, since any generating set for $P$ generates
all subgroups of $P$ as well. Thus $ \left|P : \Phi( P) \right| \geq \left| K : \Phi(K) \right|$, as required. The proof is now 
complete.
\end{proof}

\noindent In view of the above lemma, equation \eqref{doubleunion} assumes the form
\begin{equation}
\C = \bigcup\limits_{Z \leq K \leq P} \bigcup\limits_{\Phi(K) \leq H \leq Z}\U_{K,H}. \label{doubleunion2}
\end{equation}
We note in passing that the inequality of Lemma \ref{frattini} becomes an equality precisely 
when $K / Z$ is a subfield of $P / Z \cong \F_q$, that is, if and only if
$\log_2 \left| K : Z \right|$ is a divisor of $\log_2 \left| P : Z \right|$.

We shall now briefly recall some basic concepts from the theory of group extensions. We say
that the group $G$ is an extension of $N$ by $F$ if $G$ has a normal subgroup $N$ such that 
$G / N \cong F$. If $G$ is such an extension, with $\phi: F \to G/ N$ realising the 
isomorphism, then a section of $G$ through $F$ is any set $\left\{\tau(f) : f \in F\right\}$ such that $\tau(1)=1$
and $\tau(f)$ is a representative for the coset $\phi(f)$. Assuming that $N$ is abelian, the map 
$F \to \Aut(N)$, $f \mapsto \left(n \mapsto n^{\tau(f)}\right)$ is well defined and independent of $\tau$.
The following
\begin{equation*}
Z^1(F,N) \coloneqq \left\{ \gamma : F \to N \mid \gamma(f_1f_2) = \gamma(f_1)^{\tau(f_2)}\gamma(f_2),\,\ \text{for all} \,\  f_1,f_2 \in F \right\}
\end{equation*}
is known as the group of \textbf{1-Cocycles}, while
\begin{equation*}
B^1(F,N) \coloneqq \left\{ \gamma_n = \left( f \mapsto nn^{-f} \right) : F \to N \mid n \in N \right\}
\end{equation*}
is the group of \textbf{1-Coboundaries}. It is easy to see that $B^1$ is a subgroup of $Z^1$.
Provided the extension $G$ splits over $N$ and $K \leq G$ is a fixed complement, every
complement of $N$ in $G$ can be written as $\left\{ k\gamma(\overline{k}): k \in K \right\}$ for
some $\gamma \in Z^1$, and two complements corresponding to cocycles $\gamma, \delta \in Z^1$
are conjugate in $G$ if and only if $\gamma \delta^{-1}$ lies in $B^1$.

\textit{Thus the factor 
group $H^1=Z^1/B^1$ is in one-to-one correspondence to the conjugacy classes of
complements of $N$ in $G$.}

Note that if $N \leq Z(G)$, then $\gamma_n=\gamma_1$ for all $n \in N$, thus $B^1$ is
the trivial group. Moreover the group of 1-Cocycles reduces to 
\begin{equation*}
Z^1(F,N) = \left\{ \gamma : F \to N \mid \gamma(f_1f_2) = \gamma(f_1)\gamma(f_2),\,\ \text{for all} \,\  f_1,f_2 \in F \right\},
\end{equation*}
which is, by definition, the group of homomorphisms $\Hom(F,N)$. Thus, in the case of
a central subgroup $N$, one has
\begin{equation*}
H^1(F,N) \cong \Hom(F,N).
\end{equation*}
Taking $G=K / H$ and $N=Z / H$ in the above relation, and noting that $F=K/H \Big/ Z/H \cong K/Z$, yields
\begin{align}\label{hom}
\nonumber         H^1\left(K / Z,Z / H\right)	&\cong \Hom\left(K / Z,Z / H\right) \\
                                                                	&\cong \Hom\left(K / Z , Z \big/ \Phi(K) \Big/ H \big/ \Phi(K) \right).
\end{align}
Let us rewrite \eqref{doubleunion2} as
\begin{equation}\label{doubleunion3}
\C = \bigcup\limits_{K/Z \leq P/Z} \bigcup\limits_{ H / \Phi(K) \leq Z / \Phi(K)} \U_{K,H}.
\end{equation}
We notice that the factor groups $K /Z$ and $Z / H$ are elementary abelian, thus both $K /Z$ and $Z / H$ are vector spaces
over $\F_2$. Set $V \coloneqq V(2,n) \cong P/Z$, $X \coloneqq K/Z$, $V(X) \coloneqq Z / \Phi(K)$, 
and $Y \coloneqq H / \Phi(K)$ to obtain yet another expression
\begin{equation}\label{doubleunion4}
\C = \bigcup\limits_{X \subseteq V} \bigcup\limits_{Y \subseteq V(X)} \U_{X,Y},
\end{equation}
where $\U_{X,Y}$ is defined naturally in correspondence to $\U_{K,H}$. In this notation 
\begin{equation}
\Hom\left(K / Z , Z \big/ \Phi(K) \Big/ H \big/ \Phi(K) \right)         =                 \Hom \left( X , V(X) \big/ Y \right)  \cong        \Hom \left( X , Y'  \right),
\end{equation}
where $Y'$ is such that $Y \oplus Y' = V(X)$. Each element of 
$\Hom \left( X , Y'  \right)$ is a linear transformation of vector 
spaces, thus $\Hom \left( X , Y'  \right) \cong \mathcal{L}(X,Y')$. 
Since $\U_{X,Y}$ and $\mathcal{L} \left( X, Y' \right)$ are in 
bijection, equation \eqref{doubleunion4} yields
\begin{equation}\label{C}
\left| \C \right| = \sum\limits_{X \subseteq V} \sum\limits_{\substack{Y \subseteq V(X) \\ V(X) = Y \oplus Y'}} \left| \mathcal{L} \left( X, Y' \right) \right|.
\end{equation}
Of course,
\begin{equation}\label{L}
\dim \mathcal{L}(X,Y') = \dim X \dim Y',
\end{equation}
but it is important to note that the dimension of the $V(X)$-space (which specifies the range of values for the 
dimension of the $Y$-space, thus also for the dimension of the $Y'$-space), does not solely depend on $\dim X$, 
but rather on the $X$-space itself.\footnote{In general, there exist distinct subgroups $Z \leq K_1 , K_2$ of $P$ 
such that $\left| K_1 / Z \right| = \left| K_2 / Z \right|$, but $\left| \Phi(K_1)\right| \neq \left| \Phi(K_2) \right|$.} 

Now consider an element $U$ of $\U_{X,Y}$. Clearly $K=UZ$ normalises $U$, thus $P \geq N_P(U) \geq UZ$.
Since $\left| X \right| = \left| K : Z \right| = \left| UZ : Z \right| = \left| U : U \cap Z \right|$, one has 
\begin{equation}\label{nbounds}
1 \leq \left| P : N_P(U) \right| \leq \frac{\left| Z \right|^2}{\left| UZ \right|} = \frac{\left| Z \right|}{\left|U:U \cap Z\right|} = \frac{\left| Z \right|}{\left| X \right|} = \left|X'\right|,
\end{equation}
where $X'$ is such that $X \oplus X' = V$. Put informally, the size of each conjugacy class of subgroups with given ``$X$-part" is at most the size of the ``$X'$-part". 
Assembling equation \eqref{C} and inequality \eqref{nbounds} yields
\begin{equation}\label{ubound}
\left|\s( P)\right| \leq \sum\limits_{\substack{X \subseteq V \\ V = X \oplus X'}} \sum\limits_{\substack{Y \subseteq V(X) \\ V(X) = Y \oplus Y'}} \left| \mathcal{L} \left( X, Y' \right) \right| \left|X'\right|.
\end{equation}
The proof of the following lemma is now straightforward.
\begin{lemma}\label{boundP}
Let $P \in \Syl_2 \left( \Sz(q) \right)$. The number of subgroups of $P$ satisfies the following inequality
\begin{equation}
\left|\s( P)\right| \leq \sum\limits_{i=0}^{n}{n \brack i}_2 \sum\limits_{j=0}^{n-i} {n-i \brack j}_2 2^{n+i(n-(i+j+1))}. 
\end{equation}
\end{lemma}
\begin{proof}
In view of the inequality shown in Lemma \ref{frattini}, one has 
$\left| V(X) \right| \leq \left| Z \right| \left| X \right|^{-1} = \left|X'\right|$. Now let $V^{\ast}(X)$ 
be the subspace of the $X'$-space isomorphic to $V(X)$ under the 
isomorphism carrying $P/Z$ to $Z$. The right-hand-side of inequality 
\eqref{ubound} may thus be rewritten as
\begin{align*}
		       	\sum\limits_{X \subseteq V} 
                        \sum\limits_{Y \subseteq V(X)} 
                        \left| \mathcal{L} \left( X, Y' \right) \right| \left|X'\right|                                 
&= 
                	\sum\limits_{X \subseteq V} 
                	\sum	\limits_{W \subseteq V^{\ast}(X)} 
                	\left| \mathcal{L} \left( X, W' \right) \right| \left|X'\right| \\
&\leq 
                	\sum\limits_{X \subseteq V} 
                	\sum\limits_{\makebox[\sumd]{$\scriptstyle W \subseteq X'$}} 
                	\left| \mathcal{L} \left( X, W' \right) \right|\left|X'\right|, 
\end{align*}
with the understanding that the dash symbol refers to a complementary subspace.
In turn, the right-hand-side of the above inequality is
\begin{equation*} 
\sum\limits_{i=0}^{n}	\sum\limits_{	\substack	{	X \subseteq V	\\	\dim X=i	}} 
\sum\limits_{j=0}^{n-i}	\sum\limits_{	\substack	{	W \subseteq X'	\\	\dim W=j	}}	\left| \mathcal{L} \left( X, W' \right) \right| \left|X'\right|,
\end{equation*}
which, by equation \eqref{L}, is equal to
\begin{equation*}
\sum\limits_{i=0}^{n}{n \brack i}_2 
\sum\limits_{j=0}^{n-i} {n-i \brack j}_2 2^{i(n-i-j)}2^{n-i}                                                                 = 
\sum\limits_{i=0}^{n}{n \brack i}_2 
\sum\limits_{j=0}^{n-i} {n-i \brack j}_2 2^{n+i(n-(i+j+1))}.
\end{equation*} 
The proof is complete.
\end{proof}
\subsection{The subgroups of the normaliser $\Gamma = P \rtimes C$}\label{normaliser}
Recall that the multiplicative group $C = \F_q^{\times}$ of the 
field acts via automorphisms on $P$; in fact, the action of $C$ 
on the nonidentity elements of both $Z$ and $P/Z$ is regular, thus, a fortiori, 
a Frobenius action. 
\begin{lemma}
Let $B \leq C$, and suppose that both $U$ and $U^g$ are $B$-invariant subgroups of $P$, where $g \in P$. Then $g \in N_P(U)$.
\end{lemma}
\begin{proof}
First note that the $B$-invariance of $U$ implies the $B$-invariance 
of $N_P(U)$. To see why, let $b \in B$, $n \in N_P(U)$. Then 
$U^{b(n)} = b \left( U^{n} \right) = b(U) = U$, where the 
second equality holds because $n$ normalises $U$, and the last 
equality holds because $U$ is $B$-invariant. Therefore 
$b(n) \in N_P(U)$, as claimed. We infer from this that the induced 
action of $B$ on $P \big/ N_P(U) = \overline{P}$ is Frobenius.\footnote{See Isaacs~\cite{isaacs}, Corollary 6.2.} 
Now, suppose that $b$ is a nontrivial element of $B$. Then 
$U^g = b \left( U^g \right) = U^{b(g)}$, thus 
$b^{-1}(g)g \in N_P(U)$. Hence 
$b(\overline{g}) = \overline{g}$, i.e., 
$\overline{g} \in C_{\overline{P}}(b) = \overline{1} = N_P(U)$,
where the first equality holds because $b$ is nontrivial and 
the action Frobenius. The claim follows.
\end{proof}

We deduce that at most one element from each conjugacy class is 
$B$-invariant, thus we may as well consider representatives for the 
conjugacy classes of subgroups of $P$ and ask which of those 
representatives are $B$-invariant. We shall then be able to determine 
all subgroups of $\Gamma$ by observing that $U^{g^{-1}}$ is 
$B$-invariant if and only if $U$ is $B^{g}$-invariant, i.e., the conjugates 
of $U$ are acted upon by the different inverse-conjugates of $B$, where 
$U$ ranges in the set of $B$-invariant subgroups of $P$.

As mentioned previously, the action of $C$ 
on the nonidentity elements of both $Z$ and $P/Z$ is regular, thus 
Dickson's ``multiplier argument"\footnote{See Dickson~\cite{dickson}, \S 70.} 
is in effect. In particular, both $Z$ and $P/Z$ are vector spaces over 
the subfield $\F_{b}$ that $b$ generates, where $\langle b \rangle =B$ 
is any subgroup of $C$, and isomorphic to 
$V_b \coloneqq V \left( 2^{m_b}, \frac{n}{m_b}\right)$, where 
$\left|\F_b\right| = 2^{m_b}$, $m_b \coloneqq \min \left\{r \in \N : o(b) \mid 2^r-1\right\}$. 

With this in mind, let us retain the notation $V_b$ for the space 
$P/Z$ and write $\overline{V_b}$ for the $Z$-space, so that we may 
distinguish between them. Further, for each $X \subseteq V_b$ define 
$V_b(X)$ to be the $\F_b$-space $Z\big/ \Phi(K)$, where $K$ is the full 
preimage of $X$. Let $\U_{X,Y}(\F_b)$ be the full preimages of a set 
of representatives for the $P$-classes of complements to $Z/H$ in $K/H$, 
where $H$ is the full preimage of the subspace $Y \subseteq V_b(X)$. 
Similar considerations to the ones established in the first part of this section 
furnish a proof for the following lemma.
\begin{lemma}\label{boundGamma}
Let $\Gamma$ be the normaliser of a Sylow 2-subgroup $P$ of $\Sz(q)$. Then
\begin{equation}
\left|\s(\Gamma)\right| \leq \sum\limits_{b \mid q-1} \sum\limits_{i=0}^{\frac{n}{m_b}}{\frac{n}{m_b} \brack i}_{2^{m_b}} \sum\limits_{j=0}^{\frac{n}{m_b}-i} {\frac{n}{m_b}-i \brack j}_{2^{m_b}} 2^{n+i(n-m_b(i+j+1))}. 
\end{equation}
\end{lemma}
\begin{proof}
The proof is identical to that of Lemma \ref{boundP}; 
the only difference is that instead of $\F_2$, the 
underlying field now is $\F_b$. The details are thus omitted.
\end{proof}

\noindent Setting $\I( P) \coloneqq \left|\s(\Gamma)\right|-\left|\s( P)\right|$, one has
\begin{align}\label{triple}
\nonumber                 \I( P)                                                		&\leq	\sum\limits_{\substack{b \mid q-1 \\ b>1}}
                                                                                                                      	\sum\limits_{i=0}^{\frac{n}{m_b}}
                                                                                                                        {\frac{n}{m_b} \brack i}_{2^{m_b}} 
                                                                                                                        \sum\limits_{j=0}^{\frac{n}{m_b}-i}
                                                                                                                        {\frac{n}{m_b}-i \brack j}_{2^{m_b}}
                                                                                                                        2^{n+i(n-m_b(i+j+1))}\\
                                                                                                         &=		\sum\limits_{\substack{b \mid q-1 \\ b>1}} 
                                                                                                                        \sum\limits_{i=0}^{\frac{n}{m_b}}
                                                                                                                        \sum\limits_{j=0}^{\frac{n}{m_b}-i} 
                                                                                                                        {\frac{n}{m_b} \brack i}_{2^{m_b}}
                                                                                                                        {\frac{n}{m_b}-i \brack j}_{2^{m_b}} 
                                                                                                                        2^{n+i(n-m_b(i+j+1))}.
\end{align}
Note that the $q$-binomial coefficient $ {m \brack k}_q$ satisfies the elementary double inequality 
\begin{equation} \label{dbound}
q^{k(m-k)} \leq  {m \brack k}_q \leq q^{k(m-k+1)}. 
\end{equation}
To see why that must be, recall that 
\begin{equation*}  
{m \brack k}_q = \frac{ ( q^m-1)(q^{m-1}-1) \dots (q^{m-k+1}-1) } {( q^k-1)(q^{k-1}-1) \dots (q-1) } = \prod\limits_{i=0}^{k-1}\frac{q^{m-i}-1}{q^{k-i}-1}, 
\end{equation*}
and notice that for each factor in the product we have
\begin{equation*} q^{m-k} \leq \frac{q^{m-i}-1}{q^{k-i}-1} \leq q^{m-k+1}. \end{equation*}
Thus
\begin{equation*} q^{k(m-k)} = \prod\limits_{i=0}^{k-1} q^{m-k} \leq {m \brack k}_q \leq \prod\limits_{i=0}^{k-1} q^{m-k+1} = q^{k(m-k+1)}, \end{equation*}
as claimed. In view of the above upper bound, we may thus write inequality \eqref{triple} as 
\begin{align}\label{3times}
\nonumber        \I( P)                          				&\leq        \sum\limits_{\substack{b \mid q-1 \\ b>1}} 
                                                                                                        \sum\limits_{i=0}^{\frac{n}{m_{b}}} 
                                                                                                        \sum\limits_{j=0}^{\frac{n}{m_{b}}-i}
                                                                                                        2^{m_{b} i \left(\frac{n}{m_{b}}-i+1\right) } 
                                                                                                        2^{m_{b} j \left(\frac{n}{m_{b}}-i-j+1\right) }
                                                                                                        2^{ijm_{b}} 
                                                                                                        2^{n-im_{b}}\\
\nonumber                                                                        &=		\sum\limits_{\substack{b \mid q-1 \\ b>1}} 
                                                                                                        \sum\limits_{i=0}^{\frac{n}{m_{b}}} 
                                                                                                        \sum\limits_{j=0}^{\frac{n}{m_{b}}-i}
                                                                                                        2^{ i \left(n-im_{b}+m_{b}\right) } 
                                                                                                        2^{ j \left(n-im_{b}-jm_{b}+m_{b}\right) }
                                                                                                        2^{ijm_{b}} 
                                                                                                        2^{n-im_{b}}\\
                                                                                        &=		\sum\limits_{\substack{b \mid q-1 \\ b>1}} 
                                                                                                        \sum\limits_{i=0}^{\frac{n}{m_{b}}} 
                                                                                                        \sum\limits_{j=0}^{\frac{n}{m_{b}}-i} 
                                                                                                        2^{ f ( i , j , m_{ b } , n ) },
\end{align}
where 
\begin{equation*} f(i,j,m_{ b },n )                                        \coloneqq                        n(i+j+1)-m_{b}\left(i^2+j^2-j\right). \end{equation*}
The summation limits of the innermost double sum as well as the nature of the summand make it clear that the quantity
\begin{equation*}                         
\sum\limits_{i=0}^{\frac{n}{m_{b}}} \sum\limits_{j=0}^{\frac{n}{m_{b}}-i} 2^{ f ( i , j , m_{ b } , n ) },
\end{equation*}
when viewed as a function of $m_{b}$ only, attains its maximum at 
\begin{equation*} 
m_0 \coloneqq \min \left\{ m_{b} : o(b) \mid q-1, b \neq 1 \right\} = \min \left\{ p \in \mathbb{P} : p \mid n \right\}. 
\end{equation*}
Since $n$ is odd, we see that $m_0 \geq 3$. Writing $ n' \coloneqq \lfloor \frac{n}{3}\rfloor$, we obtain
\begin{equation} \label{double} 
\sum\limits_{i=0}^{\frac{n}{m_{b}}} \sum\limits_{j=0}^{\frac{n}{m_{b}}-i} 2^{ f ( i , j , m_{ b } , n ) }		\leq 
\sum\limits_{i=0}^{\frac{n}{m_0}} \sum\limits_{j=0}^{\frac{n}{m_0}-i} 2^{ f ( i , j , m_{ 0 } , n ) }		\leq 
\sum\limits_{i=0}^{n'} \sum\limits_{j=0}^{n'-i} 2^{ f ( i , j , 3 , n ) }. 
\end{equation}
Therefore, inequality \eqref{3times} becomes
\begin{equation}\label{IP}
\I( P) \leq \sum\limits_{\substack{b \mid q-1 \\ b>1}}        \sum\limits_{i=0}^{n'} \sum\limits_{j=0}^{n'-i} 2^{ f ( i , j , 3 , n ) }.
\end{equation}
In the following section we shall obtain an upper bound for the 
right-hand-side of the above inequality and use this to establish 
that $\Gamma$ and $P$ have the same 
number of subgroups asymptotically speaking.
\section{Proof of $\left| \s(\Gamma) \right| \sim \left| \s( P ) \right|$.}\label{6}
Let us fix $n$ temporarily (thus also $n'$), and define
\begin{equation*} 
\R \coloneqq \left\{ (x,y) \in \mathbb{R}^2 \mid 0 \leq x \leq n' , 0 \leq y \leq n'-x \right\} 
\end{equation*}
to be the triangular region of the Cartesian plane lying in the first quadrant 
and below the line $x+y=n'$. Moreover, let 
\begin{equation*} 
\f : \R \to \mathbb{R}, \ \ (x,y) \mapsto n(x+y+1)-3(x^2+y^2-y)
\end{equation*}
be the extension of $f$ over the reals. We shall apply standard 
techniques from calculus in order to find the (absolute) maximum 
of $\f$ in~$\R$. We begin by showing that $\f(x,y)$ has no interior 
critical points. Now,
\begin{eqnarray*}
\frac{\partial \f}{\partial x} &=& n-6x, \ \ \text{and}\\
\frac{\partial \f}{\partial y} &=& n-6y+3.
\end{eqnarray*}
At an interior critical point the partial derivatives vanish. This, in our 
case, is equivalent to
$(x_0,y_0) = \left(\frac{n}{6},\frac{n}{6} + \frac{1}{2} \right)$.
But $x_0+y_0 = \frac{n}{3} + \frac{1}{2} > n'$, which forces said candidate point 
to lie outside~$\R$. Thus $\f(x,y)$ has no interior critical points, as claimed. 

We now check the maximum value of $\f(x,y)$ on the boundary of~$\R$. 
The three cases to consider here correspond to the sides of our triangle 
and are
\begin{eqnarray*}
\f(0,y)                         &=& -3y^2+(n+3)y+n,                                \\ 
\f(x,0)                         &=& -3x^2+nx+n,                                          \\ 
\f(x,n'-x)                         &=& -6x^2+3(2n'-1)x+3n'+nn'+n-3{n'}^2,  
\end{eqnarray*}
where $x,y$ range in $[0,n']$.
In each case the function $\f$ is a quadratic polynomial $\alpha z^2 + \beta z + \gamma$. 
Since $\alpha<0$ in all cases, and because $z_0 \coloneqq -\frac{ \beta}{2 \alpha}$ is an interior 
point of the corresponding line segment, we see that $\f$ peaks at $z_0$. 
Thus the desired maximum of $\f$ is the maximum among
\begin{eqnarray*}
\f\left( 0 , \frac{n+3}{6} \right)                                                                 &=& \frac{1}{12}n^2 + \frac{3}{2}n + \frac{3}{4},                \\
\f\left( \frac{n}{6} , 0 \right)                                                                 &=& \frac{1}{12}n^2 + n,                                                        \\
\f\left(\frac{2n'-1}{4} , \frac{2n'+1}{4} \right)                                         &=& n'\left[n-\frac{3}{2}(n'-1)\right] + n + \frac{3}{8}.                        
\end{eqnarray*}
Using $\frac{n}{3}-1 \leq n' \leq \frac{n}{3}$, one easily sees that
\begin{equation*} 
\frac{n^2}{6} + 2n + \frac{3}{8} \geq \f\left(\frac{2n'-1}{4} , \frac{2n'+1}{4} \right) \geq \frac{n^2}{6} + n - \frac{9}{8}. 
\end{equation*}
Therefore
\begin{eqnarray*}
\max_{n \geq 9} \left\{ f(i,j,3,n) : (i,j) \in \R \cap \N^2 \right\}                 &\leq&         \max_{n \geq 9} \left\{ \f(x,y) : (x,y) \in \R \right\}        \\                 
                                                                                                        &\leq&        \frac{n^2}{6} + 2n + \frac{3}{8}.
\end{eqnarray*}
We may thus write
\begin{equation*}
\sum\limits_{i=0}^{n'} \sum\limits_{j=0}^{n'-i} 2^{ f ( i , j , 3 , n ) }         \leq        \sum\limits_{i=0}^{n'} \sum\limits_{j=0}^{n'-i} 2^{\frac{n^2}{6} + 2n + \frac{3}{8}}                                
                                                                                                        <        \left( \frac{n}{3}+1 \right)^2 2^{\frac{n^2}{6} + 2n + \frac{3}{8}}.        
\end{equation*}
Substituting this in \eqref{IP}, we obtain
\begin{align}\label{usedlater}
\nonumber	\I( P)         \leq				\sum\limits_{\substack{b \mid q-1 \\ b>1}} 
\nonumber                        			     		\sum\limits_{i=0}^{n'} 
\nonumber			              				\sum\limits_{j=0}^{n'-i} 
\nonumber          							2^{ f ( i , j , 3 , n ) }                                                                                                                
\nonumber                	&\leq                		(d(q-1) -1) \left( \frac{n}{3}+1 \right)^2 2^{\frac{n^2}{6} + 2n + \frac{3}{8}}		\\
\nonumber                	&<                        		2^n n^2 2^{\frac{n^2}{6} + 2n + \frac{3}{8}}								\\
					&<                        		2^{\frac{n^2}{6} + 4n + \frac{1}{2}}.                                                                                                                                                                                                        
\end{align}
This bound is sufficient for our purposes. In order to see why that is, we 
look back at \eqref{dbound}. Take $m=n$, $k=\frac{n-1}{2}$ and $q=2$ 
there. Then
\begin{equation}
2^{\frac{n^2-1}{4}} \leq  {n \brack \frac{n-1}{2}}_2 \leq 2^{\frac{n^2+2n-3}{4}}. 
\end{equation}
Since $Z$ is an elementary abelian 2-group, the quantity 
${n \brack \frac{n-1}{2}}_2$ counts the number of central subgroups 
of order $2^{\frac{n-1}{2}}$ in $P$. Hence
\begin{equation}\label{finalingredient}
2^{\frac{n^2-1}{4}} \leq  {n \brack \frac{n-1}{2}}_2 < \left| \s( P ) \right|,
\end{equation}
which in turn implies that 
\begin{equation*}
0 < \frac{\left| \s(\Gamma) \right| - \left| \s( P ) \right|}{\left| \s( P ) \right|} < 2^{\frac{n^2}{6} + 4n + \frac{3}{8} -\frac{n^2}{4}+\frac{1}{4}}.
\end{equation*}
Thus 
\begin{equation*}
\lim\limits_{n \to +\infty} \frac{\left| \s(\Gamma) \right| - \left| \s( P ) \right|}{\left| \s( P ) \right|}=0;
\end{equation*}
equivalently
\begin{equation}\label{PisequaltoGamma}
\lim\limits_{n \to +\infty} \frac{\left| \s(\Gamma) \right| } { \left| \s( P ) \right|}=1.
\end{equation}
A similar analysis to the one outlined above will reveal that 
\begin{equation}\label{boundforP}
\left| \s( P ) \right| < 2^{\frac{(n+1)^2}{2}}
\end{equation}
for all $n \in \N$, where the maximum of the implied $\overline{f}$ now occurs at an interior point.

\section{Almost all subgroups are $2$-groups}\label{7}
We begin this section with the following lemma, which is a straightforward application of the
Schur-Zassenhaus theorem.
\begin{lemma}\label{product}
Let $G=A \rtimes B$ be a finite group, where $\gcd \left( \left| A \right|,\left| B \right| \right)=1$. If $H \leq G$, then 
$H = \left( H \cap A \right) \rtimes \left( H \cap B^g \right)$ for some $g \in A$.
\end{lemma}
\begin{proof}
Observe that $H \cap A$ is a normal subgroup of $H$, and that 
$\gcd \left( \left| H \cap A \right|, \left| H \big/ H \cap A \right| \right)=1$, since
$H \big/ H \cap A$ is isomorphic to a subgroup of $B$. By the Schur-Zassenhaus 
theorem, $H \cap A$ has a complement in $H$, say $C$, thus $H = \left( H \cap A \right) C$.
Quoting the same theorem there exists a $g \in G$ such that $C^g \leq B$. Now write $g=ba$
for some $b \in B$, $a \in A$. Then $C^a \leq B$, hence
$H^a = \left( H \cap A \right)^aC^a \leq \left( H^a \cap A \right) \left( H^a \cap B \right) \leq H^a.$
We conclude that $H = \left( H \cap A \right) \left( H \cap B^{g} \right)$ for $g=a^{-1}$.
\end{proof}

\begin{corollary}
Suppose that $G$ is a finite group satisfying the conditions of Lemma \ref{product}. Then 
$\left| \s(G) \right| \leq \left| A \right| \left| \s(A) \right| \left| \s(B) \right|$.
\end{corollary}
\begin{proof}
Consider the map $f: \s(G) \to A \times \s(A) \times \s(B)$, defined 
via the rule $H \mapsto \left(g,H \cap A, H^{g^{-1}}\cap B\right)$,
where $g$ is such that $H = \left( H \cap A \right) \rtimes \left( H \cap B^g \right)$,
and observe that $f$ is injective.
\end{proof}

We apply the above corollary, along with the elementary inequality $d(k) \leq 2 \sqrt k$, to the groups $D_{2(q-1)}$, $C_{q-\theta+1} \rtimes C_4$, and $C_{q+\theta+1} \rtimes C_4$:
\begin{enumerate}
\item $\left| \s \left( D_{2(q-1)} \right) \right| \leq 2(q-1) d(q-1) \leq 4 q^{\frac{3}{2}}$,
\item $\left| \s \left( C_{q-\theta+1} \rtimes C_4 \right) \right| \leq 3(q-\theta+1) d (q-\theta+1) \leq 6q^{\frac{3}{2}}$,
\item $\left| \s \left( C_{q+\theta+1} \rtimes C_4 \right) \right| \leq 3(q+\theta+1) d (q+\theta+1) \leq 6 \cdot 2^{\frac{3}{2}}q^{\frac{3}{2}} < 17q^{\frac{3}{2}}$.
\end{enumerate}
Assuming that $n \geq 9$, we see that $\left| \s \left( H \right) \right| < q^2$ when 
$H$ is any of the groups in the above list. In fact this inequality holds for all 
$n \in \N$ by a direct calculation. We shall also require the following lemma.
\begin{lemma}\label{induction}
The number of subgroups of $\Sz(q)$ satisfies the following inequality 
\[
\left| \s \left( \Sz(q)\right) \right| < 2^{\frac{11}{5}(\log_2 q)^2},
\]
for all $q$ an odd power of 2.
\end{lemma}
\begin{proof}
The proof is by induction on the exponent of $q$. To establish the base case, 
we use a computer algebra programme to compute the size of 
the subgroup lattice of $\Sz(8)$, and find that 
$\left| \s \left( \Sz(8) \right) \right| = 17295 < 2^{15} <2^{\frac{99}{5}}$.
Now set $m \coloneqq \log_2 q$, and let $\left\{ p_1,\dots,p_k \right\}$
be the set of distinct prime divisors of $m$. Since each subgroup of $\Sz(q)$ is contained in 
one of its maximal subgroups, we see that
\begin{align*}
\left| \s \left( \Sz(q)\right) \right| 	&< (q^2+1)\left| \s(\Gamma)\right|+\frac{1}{2}q^2(q^2+1)\left| \s\left( D_{2(q-1)} \right)\right|\\
									&+\frac{1}{4}q^2(q-1)(q+\theta+1)\left| \s\left( C_{q-\theta+1} \rtimes C_4 \right)\right|\\
									&+\frac{1}{4}q^2(q-1)(q-\theta+1)\left| \s\left( C_{q+\theta+1} \rtimes C_4 \right)\right|\\
									&+\sum\limits_{i=1}^{k}\left|\Sz(q) : \Sz\left(q^{1/p_i}\right)  \right| \left|\s \left( \Sz\left(q^{1/p_i}\right) \right) \right|.
\end{align*}
Observe that $\left|\Sz(q) : \Sz\left(q^{1/p_i}\right)  \right| < q^5$, and recall that $\left| \s \left( H \right) \right| < q^2$ when $H$ is either the dihedral group, or one of the two metacyclic Frobenius groups. Hence
\begin{align}\label{almostthere}
\nonumber \left| \s \left( \Sz(q)\right) \right| 			&< 	(q^2+1)\left| \s(\Gamma)\right|+q^2\left[\frac{1}{2}q^2(q^2+1)+\frac{1}{4}q^2(q-1)(q \pm \theta+1)\right]\\
\nonumber									&	+q^5\sum\limits_{i=1}^{k}\left|\s \left( \Sz\left(q^{1/p_i}\right) \right) \right|\\
											&=	(q^2+1)\left| \s(\Gamma)\right|+q^6+q^5\sum\limits_{i=1}^{k}\left|\s \left( \Sz\left(q^{1/p_i}\right) \right) \right|.
\end{align}
The induction hypothesis yields
\begin{align*}
\left| \s \left( \Sz(q)\right) \right| 					&< 	(q^2+1)\left| \s(\Gamma)\right|+q^6+q^5\sum\limits_{i=1}^{k}2^{\frac{11}{5}(m/3)^2}\\
											&= 	(q^2+1)\left| \s(\Gamma)\right|+q^6+q^5\omega(m)2^{\frac{11}{45}m^2}.
\end{align*}
Recall that $\left| \s(\Gamma)\right|=\left| \s( P)\right|+\I( P ) < 2^{\frac{(m+1)^2}{2}} + 2^{\frac{m^2}{6} + 4m + \frac{1}{2}}$ 
by \eqref{boundforP} and \eqref{usedlater} respectively, hence 
\begin{align*}
\left| \s \left( \Sz(q)\right) \right| &< 	(2^{2m}+1)\left( 2^{\frac{(m+1)^2}{2}} + 2^{\frac{m^2}{6} + 4m + \frac{1}{2}} \right)+2^{6m}+2^{\frac{11}{45}m^2+5m+\log_2 \omega(m)}\\
								&< 	2^{2m+\frac{1}{2}}2^{\frac{(m+1)^2}{2}+\frac{m^2}{6} + 4m + \frac{1}{2}}+2^{6m}+2^{\frac{11}{45}m^2+5m+\log_2 \omega(m)}\\
								&= 	2^{\frac{2}{3}m^2+7m + \frac{3}{2}}+2^{6m}+2^{\frac{11}{45}m^2+5m+\log_2 \omega(m)}.
\end{align*}
But $\max \left\{2^{\frac{2}{3}m^2+7m + \frac{3}{2}}, 2^{6m},2^{\frac{11}{45}m^2+5m+\log_2 \omega(m)}  \right\} = 2^{\frac{2}{3}m^2+7m + \frac{3}{2}}$ for all $m \in \N$, thus
\[
\left| \s \left( \Sz(q)\right) \right| < 2^{\frac{2}{3}m^2+7m + \frac{3}{2}+\log_2 3} < 2^{\frac{11}{5}m^2},
\]
since $\frac{11}{5}m^2 > \frac{2}{3}m^2+7m + \frac{3}{2}+\log_2 3$ for all $m \geq 5$. The induction is now complete.
\end{proof}

\noindent The constant 11/5 which appears at the exponent of the upper 
bound for $\left| \s \left( \Sz(q)\right) \right|$ in Lemma \ref{induction} is 
by no means the best possible, but it is sufficient for our purposes. To see why, we look 
back at \eqref{almostthere} which, in view of Lemma \ref{induction}, yields
\begin{align}
\frac{\left| \s \left( \Sz(q)\right) \right|}{(q^2+1)\left| \s(\Gamma)\right|} 	&<	1 +	\frac	{q^6+q^5\sum\limits_{i=1}^{k}\left|\s \left( \Sz\left(q^{1/p_i}\right) \right) \right|}
																				{(q^2+1)\left| \s(\Gamma)\right|}\\
																&<	1 + 	\frac	{2^{6 \log_2 q}+2^{\frac{11}{5}(\log_2 q /3)^2 + 5 \log_2 q + \log_2 \omega(\log_2 q)}}
																				{(q^2+1)\left| \s(\Gamma)\right|}\\
																&<	1 + 	\frac	{2^{\frac{11}{45}(\log_2 q)^2 + 5 \log_2 q + \log_2 \omega(\log_2 q)+1}}
																				{(q^2+1)\left| \s(\Gamma)\right|}.
\end{align}
We recall that $\left| \s(\Gamma)\right| > \left| \s( P )\right| > 2^{\frac{(\log_2 q)^2}{4}-\frac{1}{4}}$ by inequality \eqref{finalingredient}, thus 
\[
(q^2+1)\left| \s(\Gamma)\right| > q^2\left| \s( P )\right| > 2^{\frac{(\log_2 q)^2}{4} + 2 \log_2 q - \frac{1}{4}}. 
\]
In conclusion
\begin{align*}
\frac{\left| \s \left( \Sz(q)\right) \right|}{(q^2+1)\left| \s(\Gamma)\right|} 		&< 1 +  2^{\frac{11}{45}(\log_2 q)^2 + 5 \log_2 q + \log_2 \omega(\log_2 q)+1-\left( \frac{(\log_2 q)^2}{4} + 2 \log_2 q - \frac{1}{4} \right)}\\
																	&= 1 + 2^{-\frac{1}{180}(\log_2 q)^2 + 3 \log_2 q + \log_2 \omega(\log_2 q)+\frac{5}{4}},
\end{align*}
hence 
\[
\lim\limits_{n \to + \infty} \frac{\left| \s \left( \Sz(q)\right) \right|}{(q^2+1)\left| \s(\Gamma)\right|} = 1.
\]
Since $\lim\limits_{n \to + \infty} \frac{\left| \s(\Gamma)\right|}{\left| \s( P )\right|} = 1$ by \eqref{PisequaltoGamma}, so $\lim\limits_{n \to + \infty} \frac{(q^2+1)\left| \s(\Gamma)\right|}{\left| \E_n \right|} = 1$. Therefore
\[
\lim\limits_{n \to + \infty} \frac{\left| \s \left( \Sz(q)\right) \right|}{\left| \E_n \right|} = \lim\limits_{n \to + \infty} \frac{\left| \s \left( \Sz(q)\right) \right|}{(q^2+1)\left| \s(\Gamma)\right|}\cdot \lim\limits_{n \to + \infty} \frac{(q^2+1)\left| \s(\Gamma)\right|}{\left| \E_n \right|} =1 \cdot 1 = 1.
\]
\section{Conclusions and further research}\label{8}
We have seen that $\p \left( \Sz \left( 2^{2n+1} \right)\right)$ vanishes asymptotically; 
at the same time our intuition guides us to believe that all simple groups should have low 
subgroup permutability degrees.
We make this precise in the form of a conjecture.
\begin{conjecture}
Let $G$ be a finite simple classical (or alternating) group. Then the probability that two subgroups of $G$ permute tends to 0 as~$\left| G \right| \to \infty$.
\end{conjecture}
\noindent In particular, this conjecture strengthens Problem 4.3. in T{\u{a}}rn{\u{a}}uceanu's paper~\cite{tuarnuauceanu11}, while the present paper
and the author's recent work~\cite{aivazidis} provide a partial solution. A weaker version of the above conjecture provides an interesting non-simplicity criterion, and
stems from the empirical observation that high subgroup permutability degree forces normality.
\begin{conjecture}
Let $G$ be a finite group. If $\p(G) > \p(A_5)$, then $G$ is not simple.
\end{conjecture}

Let us now focus on what structural information for $G$ can be deduced from knowledge of $\p(G)$. As 
explained in the introduction, a finite group $G$ satisfies $\p(G) = 1$ if and only if $G$ is quasi-Dedekind;
equivalently, if and only if $G$ is nilpotent modular. We can ask what happens if either of the two conditions
is dropped. 

Nilpotency of a finite group alone cannot be related to its subgroup 
permutability degree in any meaningful way. Consider the families of groups
$\left\{C_{2^{n-3}} \times Q_8\right\}_{n=4}^{+\infty}$ and $\left\{ D_{2^n} \right\}_{n=4}^{+\infty}$, where $Q_8$ 
is the quaternion group of order 8, and $ C_{2^{n-3}} $, $D_{2^n}$ are the cyclic group and dihedral group of order
$2^{n-3}$, $2^n$ respectively. In both cases the 
groups are nilpotent, non-modular for all $n \in \N_{\geq 4}$, but 
\begin{equation*}
\lim\limits_{n\to\infty}\p\left(C_{2^{n-3}} \times Q_8\right)=1 \neq 0 = \lim\limits_{n\to\infty}\p\left(D_{2^n}\right).
\end{equation*}
Indeed, in this case the groups lie at the opposite extremes of the range of values of $\p$, asymptotically speaking. 

The modular, non-nilpotent case admits a similar answer. Denote by $r_n \coloneqq p_1 p_2 \dots p_n$ the product of the first $n$
primes, and consider the families 
\begin{equation}
\left\{ C_{r_n/r_2} \times S_3 \right\}_{n=2}^{+\infty}, \ \ \text{and} \ \ \left\{ C_{r_n/2p_n} \times D_{2p_n} \right\}_{n=2}^{+\infty},
\end{equation}
where $S_3$ denotes the symmetric group on 3 letters.
Both families consist of groups that are modular and non-nilpotent, but
\begin{equation*}
\lim\limits_{n\to\infty} \p \left( C_{r_n/r_2} \times S_3 \right) = \frac{5}{6} \neq 0 = \lim\limits_{n\to\infty} \p \left( C_{r_n/2p_n} \times D_{2p_n} \right).
\end{equation*}
\noindent Finally, it seems worthwhile to have a clearer picture of the range of values that $\p$ assumes.
\begin{question}
Which rational numbers are limit points for $\p$? Do irrational limit points exist?
\end{question}
\noindent\textbf{Acknowledgements.} The author thanks I. M. Isaacs for permission to reproduce the argument that proves Claim \ref{isaacs}.

\end{document}